\documentclass{siamart0516}

\usepackage{lipsum}
\usepackage{epstopdf}
\usepackage{amsmath,amssymb}
\usepackage{amsfonts}
\usepackage{algorithm}
\usepackage{algorithmic}
\usepackage{subcaption}
\usepackage{graphicx}
\usepackage{relsize}
\usepackage{verbatim}
\usepackage{array}
\usepackage{multirow}

\ifpdf
  \DeclareGraphicsExtensions{.eps,.pdf,.png,.jpg}
\else
  \DeclareGraphicsExtensions{.eps}
\fi

\def\noprint#1{}

\newtheorem{assumption}{Assumption}
\newcommand{\eps}{\epsilon}
\newcommand{\ddd}{\delta}

\newcommand{\R}{\mathbb{R}}

\def\tto{\;{\lower 1pt \hbox{$\rightarrow$}}\kern -10pt
           \hbox{\raise 2.8pt \hbox{$\rightarrow$}}\;}

\newcommand{\TheTitle}{Behavior of Accelerated Gradient Methods Near
  Critical Points of Nonconvex Functions}
\newcommand{\RunningTitle}{Accelerated Gradient Methods in Nonconvex
  Optimization} \newcommand{\TheAuthors}{Michael O'Neill and Stephen
  J. Wright}

\headers{\RunningTitle}{\TheAuthors}

\title{{\TheTitle}\thanks{Version of \today.  \funding{This work was
      supported by NSF Awards IIS-1447449, 1628384, 1634597, and
      1740707; AFOSR Award FA9550-13-1-0138; and Subcontract 3F-30222
      from Argonne National Laboratory. Part of this work was done
      while the second author was visiting the Simons Institute for
      the Theory of Computing, and partially supported by the
      DIMACS/Simons Collaboration on Bridging Continuous and Discrete
      Optimization through NSF Award CCF-1740425.}}}

\author{
Michael O'Neill\thanks{Computer Sciences Department, University of Wisconsin, Madison, WI 53706 (\email{moneill@cs.wisc.edu}).}
\and
Stephen J. Wright\thanks{Computer Sciences Department, University of Wisconsin, Madison, WI 53706 (\email{swright@cs.wisc.edu}).}
}

\usepackage{amsopn}
\DeclareMathOperator{\diag}{diag}

\begin{document}

\maketitle

\begin{abstract}
We examine the behavior of accelerated gradient methods in smooth
nonconvex unconstrained optimization, focusing in particular on their
behavior near strict saddle points. Accelerated methods are iterative
methods that typically step along a direction that is a linear
combination of the previous step and the gradient of the function
evaluated at a point at or near the current iterate. (The previous
step encodes gradient information from earlier stages in the iterative
process.)  We show by means of the stable manifold theorem that the
heavy-ball method is unlikely to converge to strict saddle
points, which are points at which the gradient of the objective is
zero but the Hessian has at least one negative eigenvalue.  We then
examine the behavior of the heavy-ball method and other
accelerated gradient methods
in the vicinity of a strict saddle point of a nonconvex quadratic
function, showing that both methods can diverge from this point more
rapidly than the steepest-descent method.
\end{abstract}



\section{Introduction} \label{sec:intro}

We consider methods for the smooth unconstrained optimization problem
\begin{equation} \label{eq:f}
\min_{x \in \R^n} \, f(x),
\end{equation}
where $f: \R^n \to \R$ is a twice continuously differentiable function.
We say that $x^*$ is a {\em critical point} of \eqref{eq:f} if $\nabla f(x^*)=0$.
Critical points that are not local minimizers are of little interest in the context
of the optimization problem \eqref{eq:f}, so a desirable property of
any algorithm for solving \eqref{eq:f} is that it not be attracted to
such a point. Specifically, we focus on functions with {\em strict saddle
points}, that is, functions where the Hessian at each saddle point has
at least one negative eigenvalue.
%\mocomment{I don't really know if I like putting
%this here. Mostly did it because of the reviewers wanting mention of the regularity
%conditions and I figure mentioning the Hessian up front and changing ``is a smooth
%function" to ``is a twice smooth function should be explicit enough...}

Our particular interest here is in methods that use {\em gradients}
and {\em momentum} to construct steps. In many such methods, each step
is a linear combination of two components: the gradient $\nabla f$
evaluated at a point at or near the latest iterate, and a momentum term,
which is the step between the current iterate and the previous
iterate.  There are rich convergence theories for these methods in the
case in which $f$ is convex or strongly convex, along with extensive
numerical experience in some important applications.  However,
although these methods are applied frequently to nonconvex functions,
little is known from a mathematical viewpoint about their behavior in
such settings. Early results showed that a certain modified
accelerated gradient method achieves the same order of convergence on
a nonconvex problem as gradient descent \cite{ghadimi2016accelerated}
\cite{li2015accelerated} --- not a faster rate, as in the convex
setting. 

The heavy-ball method was studied in the nonconvex setting in
\cite{zavriev1993heavy}.  From an argument based on a Lyapunov
function, this work shows that heavy-ball converges to some set of
stationary points when short step sizes are used. Their result also
implies that with these shorter stepsizes, heavy-ball converges to
these stationary points with a sublinear rate, just as gradient
descent does in the nonconvex case.  Another work studied the
continuous time heavy-ball method \cite{attouch2000heavy}.  For Morse
functions (functions where all critical points have a non-singular
Hessian matrix), this paper shows that the set of initial conditions
from which heavy-ball converges to a local minimizer is an open dense
subset of $\mathbb{R}^n \times \mathbb{R}^n$. We present a similar
result for a larger class of functions, using techniques like those of
\cite{LeeSJR16a}, where the authors show that gradient descent,
started from a random initial point, converges to a strict saddle
point with probability zero.
%(Strict saddle points are critical points
%at which the Hessian has at least one negative eigenvalue.)
We show that the discrete heavy-ball method essentially shares this
property. We also study whether momentum methods can ``escape'' strict
saddle points more rapidly than gradient descent. Experience with
nonconvex quadratics indicate that, when started close to the
(measure-zero) set of points from which convergence to the saddle
point occurs, momentum methods do indeed escape more quickly.

After submission of our paper, \cite{jin2017accelerated} described a
method that combines accelerated gradient, perturbation at points with
small gradients and explicit negative curvature detection to attain a
method with worst-case complexity guarantees.

%A recent paper \cite{LeeSJR16a}
%shows that gradient descent, started from a random initial point, converges
%to a strict saddle point with probability zero. (Strict saddle point are critical points
%at which the Hessian has at least one negative eigenvalue.)  We show that the
%discrete heavy-ball method essentially shares this property. We also study
%whether momentum methods can ``escape'' strict saddle points more
%rapidly than gradient descent. Experience with nonconvex quadratics
%indicate that, when started close to the (measure-zero) set of points
%from which convergence to the saddle point occurs, momentum methods do
%indeed escape more quickly.

% \nocite{Jin_etal_17a}

\paragraph{Notation} 

For compactness, we sometimes use the notation $(y,z)$ to denote the
vector $\left[ \begin{matrix} y \\ z \end{matrix} \right]$, for $y \in
\R^n$ and $z \in \R^n$.

\section{Heavy-Ball is Unlikely to Converge to Strict Saddle Points} 
\label{sec:sm}

We show in this section that the heavy-ball method is not attracted to
strict saddle points, unless initialized in a very particular way,
that cannot occur if the starting point is chosen at random and the
algorithm is modified slightly. Following \cite{LeeSJR16a}, our proof
is based on the stable manifold theorem.

We make the following assumption throughout this section.
\begin{assumption} \label{ass:1}
The function $f: \R^n \to \R$ is $r+1$ times continuously
differentiable, for some integer $r \ge 1$, and $\nabla f$ has
Lipschitz constant $L$.
\end{assumption}
Under this assumption, the eigenvalues of the Hessian $\nabla^2
f(x^*)$ are bounded in magnitude by $L$.

The heavy-ball method is a prototypical momentum method (see
\cite{Pol87}), which proceeds as follows from a starting point $x^0$:
\begin{equation} \label{eq:hb}
x^{k+1} := x^k - \alpha \nabla f(x^k) + \beta (x^k-x^{k-1}), \quad
\makebox{with $x^{-1}=x^0$}.
\end{equation}
Following \cite{Pol87}, we can write \eqref{eq:hb} as follows:
\begin{equation} \label{eq:hb2}
\left[ \begin{matrix} x^{k+1} \\ x^k \end{matrix} \right] =
\left[ \begin{matrix} x^k - \alpha \nabla f(x^k) + \beta (x^k-x^{k-1}) \\
x^k \end{matrix} \right].
\end{equation}
Convergence for this method is known for the special case in which $f$
is a strongly convex quadratic. Denote by $m$ the positive lower bound
on the eigenvalues of the Hessian of this quadratic, and recall that
$L$ is the upper bound. For the settings
\begin{equation} \label{eq:polyak.hb}
\alpha = \frac{4}{(\sqrt{L}+\sqrt{m})^2}, \quad
\beta = \frac{\sqrt{L}-\sqrt{m}}{\sqrt{L}+\sqrt{m}}
\end{equation}
a rigorous version of the eigenvalue-based argument in
\cite[Section~3.2]{Pol87} can be applied to show R-linear convergence
with rate constant $\sqrt{\beta}$, which is approximately
$1-\sqrt{m/L}$ when the ratio $L/m$ is large. This suggests a
complexity of $O(\sqrt{L/m} \log \epsilon)$ iterations to reduce the
error $\| x^k - x^* \|$ by a factor of $\epsilon$ (where $x^*$ is the
unique solution). Such rates are typical of accelerated gradient
methods. They contrast with the $O((L/m) \log \epsilon)$ rates attained by the
steepest-descent method on such functions.

We note that the eigenvalue-based argument that is ``sketched'' by
\cite{Pol87} does not extend rigorously beyond strongly convex
quadratic functions. A more sophisticated argument based on Lyapunov
functions is needed, like the one presented for Nesterov's accelerated
gradient method in \cite[Chapter~4]{Opt4ML}. 

%% We assume throughout\sjwcomment{No - we have stronger analysis now
%%   that tolerates $\alpha$ up to nearly $4$ times as large, as in the
%%   settings above. Work this in .....}  that the step parameters lie in
%% the following interval:
%% \begin{equation} \label{eq:ab2}
%% \alpha \in (0,1/L), \quad \beta \in (0,1).
%% \end{equation}

The key to our argument for non-convergence to strict saddle points
lies in formulating the heavy-ball method as a mapping whose fixed
points are stationary points of $f$ and to which we can apply the
stable manifold theorem. Following \eqref{eq:hb2}, we define this
mapping to be
\begin{equation} \label{eq:hb.map}
G(z_1,z_2) = \left[ \begin{matrix} z_1 - \alpha \nabla f(z_1) + \beta(z_1-z_2) \\
z_1
 \end{matrix} \right], \quad (z_1,z_2) \in \R^n \times \R^n.
\end{equation}
Note that
\begin{equation} \label{eq:hb.map.D}
DG(z_1,z_2) = \left[ \begin{matrix} (1+\beta) I - \alpha \nabla^2
    f(z_1) & -\beta I \\ I & 0 \end{matrix} \right].
\end{equation}
We have the following elementary result about the relationship of
critical points for \eqref{eq:f} to fixed points for the mapping $G$.
\begin{lemma} \label{lem:fp}
If $x^*$ is a critical point of $f$, then $(z_1^*,z_2^*)=(x^*,x^*)$ is
a fixed point for $G$. Conversely, if $(z_1^*,z_2^*)$ is a fixed point
for $G$, then  $x^*=z_1^*=z_2^*$ is a critical point
for $f$.
\end{lemma}
\begin{proof}
The first claim is obvious by substitution into \eqref{eq:hb.map}. For
the second claim, we have that if $(z_1^*,z_2^*)$ is a fixed point for
$G$, then
\[
\left[ \begin{matrix} z_1^* \\ z_2^* \end{matrix} \right] =
\left[ \begin{matrix} z_1^* - \alpha \nabla f(z_1^*) + \beta(z_1^*-z_2^*) \\
z_1^*
 \end{matrix} \right],
\]
from which we have $z_2^*=z_1^*$ and $\nabla f(z_1^*)=0$, giving the
result.
\end{proof}

We now establish that $G$ is a diffeomorphic mapping, a property
needed for application of the stable manifold result.
\begin{lemma} \label{lem:diffeo}
Suppose that Assumption~\ref{ass:1} holds. Then the mapping $G$
defined in \eqref{eq:hb.map} is a $C^r$ diffeomorphism. 
\end{lemma}
\begin{proof}
We need to show that $G$ is injective and surjective, and that $G$ and
its inverse are $r$ times continuously differentiable. 

To show injectivity of $G$, suppose that
$G(x_1,x_2)=G(y_1,y_2)$. Then, we have
\begin{equation} \label{eq:inject.1}
\left[ \begin{matrix} x_1 - \alpha \nabla f(x_1) + \beta(x_1-x_2) \\
x_1
 \end{matrix} \right] =
 \left[ \begin{matrix} y_1 - \alpha \nabla f(y_1) + \beta(y_1-y_2) \\
y_1
 \end{matrix} \right].
\end{equation}
Therefore, $x_1=y_1$, and so
\begin{equation} \label{eq:inject.2}
x_1-y_1+\beta(x_1-y_1+y_2-x_2)=\alpha(\nabla f(x_1)-\nabla f(y_1)) \;\;
\Rightarrow \;\; x_2=y_2,
\end{equation}
demonstrating injectivity.
To show that $G$ is surjective, we construct its inverse $G^{-1}$
explicitly. Let $(y_1,y_2)$ be such that
\begin{equation} \label{eq:diff.1}
\left[ \begin{matrix} y_1 \\ y_2 \end{matrix} \right] =
G(z_1,z_2) = 
\left[ \begin{matrix} z_1 - \alpha \nabla f(z_1) + \beta(z_1-z_2) \\
z_1
 \end{matrix} \right],
\end{equation}
Then $z_1=y_2$. From the  first partition in \eqref{eq:diff.1}, we 
obtain $z_2 = (z_1-y_1-\alpha \nabla f(z_1) )/\beta + z_1$, which after
substitution of $z_1=y_2$ leads to
\begin{equation} \label{eq:diff.2}
\left[ \begin{matrix} z_1 \\ z_2 \end{matrix} \right] =
G^{-1} (y_1,y_2) =
\left[ \begin{matrix} y_2 \\ \frac{1}{\beta} (y_2-y_1 - \alpha \nabla f(y_2)) + y_2 \end{matrix} \right].
\end{equation}
Thus, $G$ is a bijection. Both $G$ and $G^{-1}$ are continuously
differentiable one time less than $f$, so by Assumption~\ref{ass:1},
$G$ is a $C^r$-diffeomorphism.
\end{proof}

We are now ready to state the stable manifold theorem, which provides
tools to let us characterize the set of escaping points.
\begin{theorem}[Theorem~III.7 of \cite{shub1987global}] \label{th:smt} 
Let $0$ be a fixed point for the $C^r$ local diffeomorphism $\phi : U
\rightarrow E$ where $U$ is a neighborhood of $0$ in the Banach space
$E$. Suppose that $E = E_{cs} \oplus E_u$, where $E_{cs}$ is the
invariant subspace corresponding to the eigenvalues of $D\phi(0)$
whose magnitude is less than or equal to 1, and $E_u$ is the invariant
subspace corresponding to eigenvalues of $D\phi(0)$ whose magnitude
is greater than 1. Then there exists a $C^r$ embedded
disc $W_{loc}^{cs}$ that is tangent to $E_{cs}$ at 0 called the local
stable center manifold. Additionally, there exists a neighborhood $B$
of 0 such that $\phi(W_{loc}^{cs}) \cap B \subset W_{loc}^{cs}$, and
that if $z$ is a point such that $\phi^k (z) \in B$ for all $k \ge 0$,
then $z \in W_{loc}^{cs}$.
\end{theorem}

This is a similar statement of the stable manifold theorem to the one
found in \cite{LeeSJR16a}, except that since we have to deal with
complex eigenvalues here, we emphasize that the decomposition is
between the eigenvalues whose {\em magnitude} is less than or equal to
$1$, and greater than $1$, respectively. It guarantees the existence of
a stable center manifold of dimension equal to the number of
eigenvalues of the Jacobian at the critical point that are less than
or equal to 1.

We show now that the Jacobian $DG(x^*,x^*)$ has the properties
required for application of this result, for values of $\alpha$ and
$\beta$ similar to the choices \eqref{eq:polyak.hb}. (Note that the
conditions on $\alpha$ and $\beta$ in this result hold when $\alpha
\in (0, 4/L)$ and $\beta \in (-1+\alpha L/2,1)$, where $L$ is the
Lipschitz constant from Assumption~\ref{ass:1}.) For purposes of this
and future results in this section, we assume that at the point $x^*$
we have $\nabla f(x^*)=0$ and that the eigenvalue decomposition of
$\nabla^2 f(x^*)$ can be written as
\begin{equation} \label{eq:vdvt}
\nabla^2 f(x^*) = V \Lambda V^T = \sum_{i=1}^n \lambda_i v_i (v_i)^T,
\end{equation}
where the eigenvalues $\lambda_1,\lambda_2,\dotsc,\lambda_n$
have
\begin{equation} \label{eq:np}
\lambda_1 \ge \lambda_2 \ge \dotsc \ge \lambda_{n-p} \ge 0 >
\lambda_{n-p+1} \ge \dotsc \ge \lambda_n,
\end{equation}
for some $p$ with $1 \le p < n$, where $\Lambda = \diag
(\lambda_1,\lambda_2, \dotsc, \lambda_n)$, and where $v_i$,
$i=1,2,\dotsc,n$ are the orthonormal set of eigenvectors that
correspond to the eigenvalues in \eqref{eq:np}. The matrix $V = [
  v_1 \, | \, v_2 \, | \, \dotsc \, | \, v_n]$ is orthogonal.

\begin{theorem} \label{th:DG}
Suppose that Assumption~\ref{ass:1} holds.  Let $x^*$ be a critical
point for $f$ at which $\nabla^2 f(x^*)$ has $p$ negative eigenvalues,
where $p \ge 1$. Consider the mapping $G$ defined by \eqref{eq:hb.map}
where 
\[
0 < \alpha < \frac{4}{\lambda_1}, \quad
\beta \in \left( \max \left( -1+\frac{\alpha \lambda_1}{2},0 \right) , 1 \right),
\]
where $\lambda_1$ is the largest positive eigenvalue of $\nabla^2
f(x^*)$. Then there are matrices $\tilde{V}_s \in \R^{2n \times
  (2n-p)}$ and $\tilde{V}_u \in \R^{2n \times p}$ such that (a) the
$2n \times 2n$ matrix $\tilde{V} = \left[ \tilde{V}_s \, | \,
  \tilde{V}_u \right]$ is nonsingular; (b) the columns of
$\tilde{V}_s$ span an invariant subspace of $DG(x^*,x^*)$
corresponding to eigenvalues of $DG(x^*,x^*)$ whose magnitude is less
than or equal to $1$; (c) the columns of $\tilde{V}_u$ span an
invariant subspace of $DG(x^*,x^*)$ corresponding to eigenvalues of
$DG(x^*,x^*)$ whose magnitude is greater than $1$.
\end{theorem}
\begin{proof}
Since
\begin{equation} \label{eq:DG*}
DG(x^*,x^*) = \left[ \begin{matrix} (1+\beta) I - \alpha \nabla^2
    f(x^*) & -\beta I \\ I & 0 \end{matrix} \right],
\end{equation}
we have from \eqref{eq:vdvt} that
\[
\left[ \begin{matrix} V^T &  0 \\ 0 & V^T \end{matrix} \right] 
DG(x^*,x^*) 
\left[ \begin{matrix} V &  0 \\ 0 & V \end{matrix} \right]  =
\left[ \begin{matrix} (1+\beta)I - \alpha \Lambda & -\beta I \\
I & 0 \end{matrix} \right].
\]
By performing a symmetric permutation $P$ on this matrix, interleaving
rows/columns from the first block with rows/columns from the second
block, we obtain a block diagonal matrix with $2 \times 2$ blocks of
the following form on the diagonals, that is, 
\begin{equation} \label{eq:GMi}
P^T 
\left[ \begin{matrix} V^T &  0 \\ 0 & V^T \end{matrix} \right] 
DG(x^*,x^*) 
\left[ \begin{matrix} V &  0 \\ 0 & V \end{matrix} \right]
P =
\left[ \begin{matrix} M_1 &&& \\
&M_2 && \\
&& \ddots & \\
&&& M_n \end{matrix} \right],
\end{equation}
where
\begin{equation} \label{eq:evi}
M_i := \left[ \begin{matrix} (1+\beta) - \alpha \lambda_i & -\beta \\
1 & 0 \end{matrix} \right], \quad i=1,2,\dotsc,n.
\end{equation}
The eigenvalues of $M_i$ are obtained from the following quadratic in
$\mu$:
\begin{equation} \label{eq:ev5}
t(\mu) := ((1+\beta) - \alpha \lambda_i - \mu)(-\mu) + \beta = 0,
\end{equation}
that is,
\begin{equation} \label{eq:ev6}
t(\mu) = \mu^2 - (1+\beta-\alpha \lambda_i) \mu + \beta = 0,
\end{equation}
for which the roots are
\begin{equation} \label{eq:roots}
\mu_i^{\text{hi,lo}} = \frac12 \left[  (1+\beta-\alpha \lambda_i) \pm
\sqrt{ (1+\beta-\alpha \lambda_i)^2 - 4 \beta} \right].
\end{equation}

We examine first the matrices $M_i$ for which $\lambda_i<0$.  We
have
\[
(1+\beta - \alpha \lambda_i)^2 - 4 \beta = (1-\beta)^2 - 2\alpha 
\lambda_i (1+\beta) + \alpha^2 | \lambda_i|^2 >0,
\]
so both roots in \eqref{eq:roots} are real. Since $t(\cdot)$ is convex
quadratic, with $t(0) = \beta>0$ and $t(1) = \alpha \lambda_i <0$,
one root is in $(0,1)$ and the other is in $(1,\infty)$.  We can thus 
write
\begin{subequations} \label{eq:Mi.S}
\begin{align}
\label{eq:Mi.S.1}
M_i & = S_i  \Lambda_i S_i^{-1},
\quad \mbox{where} \\
\Lambda_i = \left[ \begin{matrix} \mu_i^{\text{hi}} & 0 \\ 0 & \mu_i^{\text{lo}} \end{matrix} \right],
\quad
S_i &=  \left[ \begin{matrix}  \mu_i^{\text{hi}}     & 1                                                        \\ 1                 & \frac{1}{\mu_i^{\text{lo}}} \end{matrix} \right], \quad
S_i^{-1} = \left( \frac{\mu_i^{\text{hi}}}{\mu_i^{\text{lo}}} - 1 \right)^{-1} 
\left[ \begin{matrix} \frac{1}{\mu_i^{\text{lo}}} & -1                                                       \\ -1               & \mu_i^{\text{hi}} \end{matrix} \right].
\end{align}
\end{subequations}
where $\mu_i^{\text{hi}}$ is the eigenvalue of $M_i$ in the range
$(1,\infty)$ and $\mu_i^{\text{lo}}$ is the eigenvalue of $M_i$ in the range
$(0,1)$. (This claim can be verified by direct calculation of the product
\eqref{eq:Mi.S.1}.)

Consider now the matrices $M_i$ for which $\lambda_i = 0$.  From
\eqref{eq:roots}, we have that the roots are $1$ and $\beta$, which
are distinct, since $\beta \in (0,1)$. The eigenvalue decompositions
of these matrices have the form
\begin{equation} \label{eq:Mi.S2}
M_i = S_i \Lambda_i S_i^{-1}, \quad \mbox{where $\Lambda_i = \diag (1,\beta)$,}
\end{equation}
and the $S_i$ are $2 \times 2$ nonsingular matrices.

When $\lambda_i>0$, we show that the eigenvalues of $M_i$ both have
magnitude less than $1$, under the given conditions on $\alpha$ and
$\beta$. Both roots in \eqref{eq:roots} are complex exactly when the
term under the square root is negative, and in this case the magnitude
of both roots is 
\[
\frac12 \sqrt{ (1+\beta-\alpha \lambda_i)^2 + \left(4 \beta -  (1+\beta-\alpha \lambda_i)^2 \right)} = \sqrt{\beta},
\]
which is less than $1$ by assumption. When both roots are real, we
have $(1+\beta-\alpha \lambda_i)^2 - 4 \beta \ge 0$, and we require
the following to be true to ensure that both are less than $1$ in
absolute value:
\begin{equation} \label{eq:roots.p.real}
-2 < (1+\beta-\alpha \lambda_i) \pm \sqrt{ (1+\beta-\alpha
  \lambda_i)^2 - 4 \beta} <2.
\end{equation}
We deal with the right-hand inequality in \eqref{eq:roots.p.real}
first. By rearranging, we show that this is implied by the 
following sequence of equivalent inequalities:
\begin{alignat*}{2}
&& (1+\beta-\alpha \lambda_i) + \sqrt{ (1+\beta-\alpha
  \lambda_i)^2 - 4 \beta} &<2 \\
& \Leftrightarrow & \sqrt{(1+\beta-\alpha  \lambda_i)^2 - 4 \beta} & < 1-\beta + \alpha \lambda_i \\
& \Leftrightarrow & (1+\beta-\alpha  \lambda_i)^2 - 4 \beta & <
(1-\beta + \alpha \lambda_i)^2 \\
& \Leftrightarrow &  \beta^2 + 2\beta(1-\alpha \lambda_i) + (1-\alpha \lambda_i)^2 - 4 \beta & < \beta^2 - 2 \beta (1+\alpha \lambda_i) + (1+\alpha \lambda_i)^2  \\
& \Leftrightarrow &  2\beta - 2 \beta \alpha \lambda_i - 4 \beta + (1-\alpha \lambda_i)^2 & < -2\beta - 2 \beta \alpha \lambda_i + (1+\alpha \lambda_i)^2 \\
& \Leftrightarrow &  (1-\alpha \lambda_i)^2 & < (1+\alpha \lambda_i)^2 \\
& \Leftrightarrow & -2\alpha \lambda_i & < 2\alpha \lambda_i,
\end{alignat*}
where the last is clearly true, because of $\alpha>0$ and
$\lambda_i>0$. Thus the right-hand inequality in
\eqref{eq:roots.p.real} is satisfied. 

For the left-hand inequality in \eqref{eq:roots.p.real}, we have
\begin{alignat*}{2}
&& -2 & < (1+\beta-\alpha \lambda_i) - \sqrt{ (1+\beta-\alpha
  \lambda_i)^2 - 4 \beta} \\
& \Leftrightarrow & -3-\beta+\alpha \lambda_i & <   - \sqrt{ (1+\beta-\alpha
  \lambda_i)^2 - 4 \beta} \\
& \Leftrightarrow & 3+\beta-\alpha \lambda_i & >  \sqrt{ (1+\beta-\alpha
  \lambda_i)^2 - 4 \beta} \\
& \Leftrightarrow & (3+\beta-\alpha \lambda_i)^2 & >   (1+\beta-\alpha
  \lambda_i)^2 - 4 \beta \\
& \Leftrightarrow &  \beta^2 + 2 \beta (3-\alpha \lambda_i) + (3-\alpha \lambda_i)^2  &> \beta^2 + 2 \beta (1-\alpha \lambda_i) + (1-\alpha \lambda_i)^2 - 4 \beta \\
& \Leftrightarrow & 6 \beta - 2 \beta \alpha \lambda_i + 9 - 6 \alpha \lambda_i + \alpha^2 (\lambda_i)^2 &>
-2\beta - 2 \beta \alpha \lambda_i + 1 - 2 \alpha \lambda_i + \alpha^2 (\lambda_i)^2  \\
& \Leftrightarrow & 8 \beta + 8 - 4 \alpha \lambda_i &>0 \\
& \Leftrightarrow &  \beta &> -1 + \alpha \lambda_i/2,
\end{alignat*}
and the last condition holds because of the assumption that $\beta
> -1 + \alpha \lambda_1/2$. This completes our proof of the claim
\eqref{eq:roots.p.real}. Thus our assumptions on $\alpha$ and $\beta$
suffice to ensure that both eigenvalues of $M_i$ defined in
\eqref{eq:evi} have magnitude less than $1$ when $\lambda_i>0$.

By defining
\[
S := \left[ \begin{matrix} I &&&&& \\
& \ddots &&&& \\
&& I &&& \\
&&& S_{n-p+1} && \\
&&&& \ddots & \\
&&&&& S_n
 \end{matrix} \right],
\]
where $S_i$, $i=n-p+1,\dotsc,n$ are the matrices defined in
\eqref{eq:Mi.S}, we have from \eqref{eq:GMi} that
\begin{align} 
\nonumber
S^{-1} P^T 
\left[ \begin{matrix} V^T &  0 \\ 0 & V^T \end{matrix} \right] 
& DG(x^*,x^*) 
\left[ \begin{matrix} V &  0 \\ 0 & V \end{matrix} \right]
P S \\
\label{eq:GMi2}
& =
\left[ \begin{matrix} M_1 &&&&& \\
& \ddots &&&& \\
&& M_{n-p} &&& \\
&&& \Lambda_{n-p+1} && \\
&&&& \ddots & \\
&&&&& \Lambda_{n} \end{matrix} \right].
\end{align}
We now define another $2n$-dimensional permutation matrix $\tilde{P}$
that sorts the entries of the diagonal matrices $\Lambda_i$, $i=n-p+1,
\dotsc,n$ into those whose magnitude is greater than one and those
whose magnitude is less than or equal to one, to obtain
\begin{align} 
\nonumber 
\tilde{P}^T S^{-1} P^T 
\left[ \begin{matrix} V^T &  0 \\ 0 & V^T \end{matrix} \right] 
& DG(x^*,x^*) 
\left[ \begin{matrix} V &  0 \\ 0 & V \end{matrix} \right]
P S \tilde{P} \\
\label{eq:GMi3}
& =
\left[ \begin{matrix} M_1 &&&& \\
& \ddots &&& \\
&& M_{n-p} && \\
&&& \tilde{\Lambda}^{\text{lo}} & \\
&&&& \tilde{\Lambda}^{\text{hi}} \end{matrix} \right],
\end{align}
where 
\[
\tilde{\Lambda}^{\text{lo}} = \diag (\mu_{n-p+1}^{\text{lo}}, \mu_{n-p+2}^{\text{lo}}, \dotsc, \mu_n^{\text{lo}}), 
\quad
\tilde{\Lambda}^{\text{hi}} = \diag (\mu_{n-p+1}^{\text{hi}}, \mu_{n-p+2}^{\text{hi}}, \dotsc, \mu_n^{\text{hi}}).
\]
We now define
\[
\tilde{V} = \left[ \begin{matrix} V &  0 \\ 0 & V \end{matrix} \right]
P S \tilde{P},
\]
which is a nonsingular matrix, by nonsingularity of $S$ and
orthogonality of $V$, $P$, and $\tilde{P}$. As in the statement of the
theorem, we define $\tilde{V}_s$ to be the first $2n-p$ columns of
$\tilde{V}$ and $\tilde{V}_u$ to be the last $p$ columns. These define
invariant spaces. For the stable space, we have
\[
DG(x^*,x^*) \tilde{V}_s = \tilde{V}_s \tilde{\Lambda}_s, \quad
\mbox{where} \;\; \tilde{\Lambda}_s := 
\left[ \begin{matrix} M_1 &&& \\
& \ddots && \\
&& M_{n-p} & \\
&&& \tilde{\Lambda}^{\text{lo}}  \end{matrix} \right],
\]
where all eigenvalues of $\tilde{\Lambda}_s$ have magnitude less than
or equal to $1$. For the unstable space, we have
\[
DG(x^*,x^*) \tilde{V}_u = \tilde{V}_u \tilde{\Lambda}^{\text{hi}},
\]
where $\tilde{\Lambda}^{\text{hi}}$ is a diagonal matrix with all diagonal
elements greater than $1$.
\end{proof}

We find a basis for the eigenspace that corresponds to the eigenvalues
of $\\DG(x^*,x^*)$ that are greater than $1$ (that is, the column
space of $\tilde{V}_u$) in the following result.
\begin{corollary} \label{cor:uspace}
Suppose that the assumptions of Theorem~\ref{th:DG} hold. Then the
eigenvector of $DG(x^*,x^*)$ that corresponds to the unstable
eigenvalue $\mu_i^{\text{hi}} >1$, $i=n-p+1,\dotsc,n$ defined in
\eqref{eq:roots} is
\begin{equation} \label{eq:vis}
\left[ \begin{matrix} v_i \\ (1/\mu_i^{\text{hi}}) v_i \end{matrix} \right],
\end{equation}
where $v_i$ is an eigenvector of $\nabla^2 f(x^*)$ that corresponds
to $\lambda_i<0$. The set of such vectors forms an orthogonal basis
for the subspace of $\R^{2n}$ corresponding to the eigenvalues of
$DG(x^*,x^*)$ whose magnitude is greater than $1$.
\end{corollary}
\begin{proof}
We have from \eqref{eq:DG*} that
\begin{align*}
DG(x^*,x^*) \left[ \begin{matrix} v_i \\ (1/\mu_i^{\text{hi}}) v_i \end{matrix} \right] & =
\left[ \begin{matrix} (1+\beta) I - \alpha \nabla^2
    f(x^*) & -\beta I \\ I & 0 \end{matrix} \right]
\left[ \begin{matrix} v_i \\ (1/\mu_i^{\text{hi}}) v_i \end{matrix} \right] \\
& =
\left[ \begin{matrix} ((1+\beta-\alpha \lambda_i) - \beta/\mu_i^{\text{hi}}) v_i \\
v_i \end{matrix} \right],
\end{align*}
so the result holds provided that
\[
(1+\beta-\alpha \lambda_i) - \beta/\mu_i^{\text{hi}} = \mu_i^{\text{hi}}.
\]
But this is true because of \eqref{eq:ev6}, so \eqref{eq:vis} is an
eigenvector of $DG(x^*,x^*)$ corresponding to the eigenvalue
$\mu_i^{\text{hi}}$. Since the vectors $\{ v_i \, | \, i=n-p+1, \dotsc,n
\}$ form an orthogonal set, so do the vectors \eqref{eq:vis} for $i=n-p+1,
\dotsc,n$, completing the proof.
\end{proof}

Our next result is similar to \cite[Theorem~4.1]{LeeSJR16a}. It is for
a modified version of the heavy-ball method in which the initial value
for $x^{-1}$ is perturbed from its usual choice of $x^0$.
\begin{theorem} \label{th:sm.pert}
Suppose that the assumptions of Theorem~\ref{th:DG} hold. Suppose that
the heavy-ball method is applied from an initial point of
$(x^0,x^{-1}) = (x^0,x^0+\epsilon y)$, where $x^0$ and $y$ are random
vectors with $i.i.d.$ elements, and $\epsilon>0$ is small. We then
have
\[
\Pr \left(\lim_k x^k = x^* \right)  = 0,
\]
where the probability is taken over the starting vectors $x^0$ and
$y$.
\end{theorem}
\begin{proof}
Our proof tracks that of \cite[Theorem~4.1]{LeeSJR16a}. As there, we
define the stable set for $x^*$ to be
\begin{equation} \label{eq:Ws}
W^s (x^*) := \left\{ (x^0,x^{-1}) \, : \, \lim_{k \to \infty} G^k(x^0,x^{-1}) = (x^*,x^*)
\right\}.
\end{equation}
For the neighborhood $B$ of $(x^*,x^*) \in \R^{2n}$ promised by
Theorem~\ref{th:smt}, we have for all $z \in W^s (x^*)$ that there is
some $l \ge 0$ such that $G^t(z) \in B$ for all $t \ge l$, and
therefore by Theorem~\ref{th:smt} we must have $G^l(z) \in
W^{cs}_{loc} \cap B$. Thus $W^s(x^*)$ is the set of points $z$ such
that $G^l(z) \in W^{cs}_{loc}$ for some finite $l$. From
Theorem~\ref{th:smt}, $ W^{cs}_{loc}$ is tangent to the subspace
$E_{cs}$ at $(x^*,x^*)$, and the dimension of $E_{cs}$ is $2n-p$, by
Theorem~\ref{th:DG} (since $E_{cs}$ is the space spanned by the
columns of $\tilde{V}_s$). This subspace has measure zero in
$\R^{2n}$, since $p \ge 1$. Since diffeomorphisms map sets of measure
zero to sets of measure zero, and countable unions of measure zero
sets have measure zero, we conclude that $W^s(x^*)$ has measure
zero. Thus the initialization strategy we have outlined produces a
starting vector in $W^s(x^*)$ with probability zero.
\end{proof}

Theorem~\ref{th:sm.pert} does not guarantee that once the iterates
leave the neighborhood of $x^*$, they never return. It does not
exclude the possibility that the sequence $\{ (x^{k+1},x^k) \}$
returns infinitely often to a neighborhood of $(x^*,x^*)$.

We note that the tweak of taking $x^{-1}$ slightly different from
$x^0$ does not affect practical performance of the heavy-ball method,
and has in fact been proposed before \cite{zavriev1993heavy}.  It also
does not disturb the theory that exists for this method, which for the
case of quadratic $f$ discussed in \cite{Pol87} rests on an argument
based on the eigendecomposition of the (linear) operator $DG$, which
is not affected by the modified starting point. We note too that the
accelerated gradient methods to be considered in the next section
can also allow $x^{-1} \ne x^0$ without
significantly affecting the convergence theory. A
Lyapunov-function-based convergence analysis of this method (see, for
example \cite[Chapter~4]{Opt4ML}, based on arguments in \cite{Tse08b})
requires only trivial modification to accommodate $x^{-1} \ne x^0$.

For the variant of heavy-ball method in which $x^0=x^{-1}$, we could
consider a random choice of $x^0$ and ask whether there is zero
probability of $(x^0,x^0)$ belonging to the measure-zero set
$W^s(x^*)$ defined by \eqref{eq:Ws}. The problem is of course that
$(x^0,x^0)$ lies in the $n$-dimensional subspace $Y^n := \{ (z_1,z_1)
\, | \, z_1 \in \R^n \}$, and we would need to establish that the
intersection $W^s(x^*) \cap Y^n$ has measure zero in $Y^n$. In other
words, we need that the set $\{ z_1 \, | \, (z_1,z_1) \in W^s(x^*) \}$
has measure zero in $\R^n$. We have a partial result in this regard,
pertaining to the set $W_{loc}^{cs}$, which is the local counterpart
of $W^s(x^*)$. This result also makes use of the subspace $E_{cs}$,
defined as in Theorem~\ref{th:smt}, which is the invariant subspace
corresponding to eigenvalues of $DG(x^*,x^*)$ whose magnitudes are
less than or equal to one.
\begin{theorem} \label{th:x0x0}
Suppose that the assumptions of Theorem~\ref{th:DG} hold. Then any
vector of the form $(w,w)$ where $w \in \R^n$ lies in the stable
subspace $E_{cs}$ only if $w \in \mbox{span} \{ v_1,v_2, \dotsc,
v_{n-p} \}$ that is, the span of eigenvectors of $\nabla^2 f(x^*)$
that correspond to nonnegative eigenvalues of this matrix.
\end{theorem}
\begin{proof}
We write $w = \sum_{i=1}^n \tau_i v_i$ for some coefficients
$\tau_i$, $i=1,2,\dotsc,n$, and show that $\tau_i=0$ for
$i=n-p+1,\dotsc,n$.

We first show that
\begin{equation} \label{eq:hg1}
DG(x^*,x^*)^k \left[ \begin{matrix} w \\ w \end{matrix} \right] =
\left[ \begin{matrix} \sum_{i=1}^n  \sigma_{k,i}  v_i \\
\sum_{i=1}^n  \eta_{k,i} v_i
\end{matrix} \right],
\end{equation}
where $\sigma_{0,i}=\tau_i$ and $\eta_{0,i}=\tau_i$, $i=1,2,\dotsc,n$.  To derive
recurrences for $\sigma_{k,i}$ and $\eta_{k,i}$, we consider the
multiplication by $DG(x^*,x^*)$ that takes us from stages $k$ to
$k+1$. We have
\begin{align*}
\left[ \begin{matrix} \sum_{i=1}^n  \sigma_{k+1,i} v_i \\
\sum_{i=1}^n  \eta_{k+1,i} v_i
\end{matrix} \right] 
& =
\left[ \begin{matrix} (1+\beta) I - \alpha \nabla^2
    f(x^*) & -\beta I \\ I & 0 \end{matrix} \right]
 \left[ \begin{matrix} \sum_{i=1}^n  \sigma_{k,i} v_i \\
\sum_{i=1}^n  \eta_{k,i}  v_i
\end{matrix} \right]  \\
&= \left[ \begin{matrix} \sum_{i=1}^n  (1-\alpha \lambda_i) \sigma_{k,i} v_i + \beta \sum_{i=1}^n (\sigma_{k,i}-\eta_{k,i})  v_i  \\
\sum_{i=1}^n  \sigma_{k,i} v_i
\end{matrix} \right].
\end{align*}
By matching terms, we have
\[
\left[ \begin{matrix} \sigma_{k+1,i} \\ \eta_{k+1,i} \end{matrix} \right] =
\left[ \begin{matrix} (1+\beta - \alpha \lambda_i) & -\beta \\
1 & 0 \end{matrix} \right]
\left[ \begin{matrix} \sigma_{k,i} \\ \eta_{k,i} \end{matrix} \right] =
M_i \left[ \begin{matrix} \sigma_{k,i} \\ \eta_{k,i} \end{matrix} \right],
\]
where $M_i$ is defined in \eqref{eq:evi}.  Using the factorization
\eqref{eq:Mi.S}, we have
\[
\left[ \begin{matrix} \sigma_{k,i} \\ \eta_{k,i} \end{matrix} \right] =
M_i^k \left[ \begin{matrix} \sigma_{0,i} \\ \eta_{0,i} \end{matrix} \right] =
S_i \Lambda_i^k S_i^{-1} \left[ \begin{matrix} 1 \\ 1 \end{matrix} \right] \tau_i,
\]
By substitution from \eqref{eq:Mi.S}, we obtain
\[
\left[ \begin{matrix} \sigma_{k,i} \\ \eta_{k,i} \end{matrix} \right] =
\left[ \begin{matrix}  \mu_i^{\text{hi}}     & 1                                                        \\ 1                 & \frac{1}{\mu_i^{\text{lo}}} \end{matrix} \right]
\left[ \begin{matrix} (\mu_i^{\text{hi}})^k & 0 \\ 0 & (\mu_i^{\text{lo}})^k \end{matrix} \right]
\left[ \begin{matrix} 1-\mu_i^{\text{lo}} \\
\mu_i^{\text{lo}} (\mu_i^{\text{hi}}-1) \end{matrix} \right] \frac{\tau_i}{\mu_i^{\text{hi}}-\mu_i^{\text{lo}}}.
\]
Because $0 < \mu_i^{\text{lo}} <1<\mu_i^{\text{hi}}$, it follows from this formula that
\[
\tau_i \neq 0 \;\; \Rightarrow \;\; \frac{\sigma_{k,i}}{\tau_i}  \to_k \infty, \;\; \frac{\eta_{k,i}}{\tau_i}  \to_k \infty,
\] 
so if $w$ has any component in the span of $v_i$, $i=n-p+1,\dotsc,n$
(that is, if $\tau_i \neq 0$), repeated multiplications of
$\left[ \begin{matrix} w \\ w \end{matrix} \right]$ by $DG(x^*,x^*)$
will lead to divergence, so $\left[ \begin{matrix} w \\ w \end{matrix}
  \right]$ cannot be in the  subspace $E_{cs}$.
\end{proof}

A consequence of this theorem is that for a random choice of $x^0$,
there is probability zero that $(x^0-x^*,x^0-x^*) \in E_{cs}$, which is
tangential to $W^{cs}_{loc}$ at $x^*$. Thus for $x^0$ close to $x^*$,
there is probability zero that $(x^0,x^0)$ is in the measure-zero set
$W^{cs}_{loc}$. Successive iterations of \eqref{eq:hb} are locally
similar to repeated multiplications of $(x^0-x^*,x^0-x^*)$ by the
matrix $DG(x^*,x^*)$, that is, for $(x^{k+1}-x^*,x^k-x^*)$ small, we
have
\[
\left[ \begin{matrix} x^{k+1}-x^* \\ x^{k} - x^*  \end{matrix} \right]
\approx
DG(x^*,x^*) 
\left[ \begin{matrix} x^{k}-x^* \\ x^{k-1} - x^*  \end{matrix} \right] \approx
DG(x^*,x^*)^{k+1}
\left[ \begin{matrix} x^{0}-x^* \\ x^{0} - x^*  \end{matrix} \right].
\]
Under the probability-one event that $x^0-x^* \notin E_{cs}$, this
suggests divergence of the iteration \eqref{eq:hb} away from
$(x^*,x^*)$.

On the other hand, we can show that if the sequence passes
sufficiently close to a point $(x^*,x^*)$ such that $x^*$ satisfies
second-order sufficient conditions to be a solution of \eqref{eq:f},
it subsequently converges to $(x^*,x^*)$.  For this result we need the
following variant of the stable manifold theorem.
\begin{theorem}[Theorem~III.7 of \cite{shub1987global}] \label{th:smt.s} 
Let $0$ be a fixed point for the $C^r$ local diffeomorphism $\phi : U
\rightarrow E$ where $U$ is a neighborhood of $0$ in the Banach space
$E$. Suppose that $E_{s}$ is the invariant subspace corresponding to
the eigenvalues of $D\phi(0)$ whose magnitude is strictly less than
1. Then there exists a $C^r$ embedded disc $W_{loc}^{s}$ that is
tangent to $E_{s}$ at $0$, and a neighborhood $B$ of $0$ such that
$W_{loc}^s \subset B$, and for all $z \in W_{loc}^{s}$, we have
$\phi^k(z) \to 0$ at a linear rate.
\end{theorem}

When $x^*$ satisfies second-order conditions for \eqref{eq:f}, all
eigenvalues of $\nabla^2 f(x^*)$ are strictly positive. It follows
from the proof of Theorem~\ref{th:DG} that under the assumptions of
this theorem, all eigenvalues of $DG(x^*,x^*)$ have magnitude strictly
less than $1$. Thus, the invariant subspace $E_s$ in
Theorem~\ref{th:smt.s} is the full space (in our case, $\R^{2n}$), so
$W_{loc}^s$ is a neighborhood of $(x^*,x^*)$. It follows that there is
some $\epsilon>0$ such that if $\| (x^{K+1},x^K) - (x^*,x^*) \| <
\epsilon$ for some $K$, the sequence $(x^{k+1},x^k)$ for $k \ge K$
converges to $(x^*,x^*)$ at a linear rate.

\section{Speed of Divergence on a Toy Problem}
\label{sec:simplest}

In this section, we investigate the rate of divergence of an
accelerated method on a simple nonconvex objective function, the
quadratic with $n=2$ defined by
\begin{equation} \label{eq:q}
f(x) = \frac12 (x_1^2 - \delta x_2^2), \quad \makebox{where $0<\delta \ll 1$.}
\end{equation}
Obviously, this function is unbounded below with a saddle point at
$(0,0)^T$. Its gradient has Lipschitz constant $L=1$.  Despite being a
trivial problem, it captures the behavior of gradient algorithms near
strict saddle points for indefinite quadratics of arbitrary dimension,
as is apparent from the analysis below.

%% For strongly convex functions $f$ for which
%% \begin{equation} \label{eq:mL}
%% mI \preceq \nabla^2 f(x) \preceq LI, \quad 0<m\le L,
%% \end{equation}
%% the analysis in \cite{Pol87} indicates that optimal values of the
%% steplength parameters in the heavy-ball method \eqref{eq:hb} are
%% approximately
%% \begin{equation} \label{eq:ab}
%% \alpha  \approx \frac{4}{L}, \quad
%% \beta \approx \frac{\sqrt{L}-\sqrt{m}}{\sqrt{L}+\sqrt{m}}.
%% \end{equation}
%% These formulae suggest ``standard'' choices for $\alpha$ and
%% $\beta$. Usually, we choose $\alpha$ to be positive with size
%% $O(1/L)$, while $\beta$ is a little less than $1$, say $0.8$ or $0.9$.

We have described the heavy-ball method in \eqref{eq:hb}.
The steepest-descent method, by contrast, takes steps of the form
\begin{equation} \label{eq:sd}
x^{k+1} = x^k - \alpha_k \nabla f(x^k),
\end{equation}
for some $\alpha_k>0$. When $\nabla f(x)$ has Lipschitz constant $L$,
the choice $\alpha_k \equiv 1/L$ leads to decrease in $f$ at each
iteration that is consistent with convergence of $\| \nabla f(x^k) \|$
to zero at a sublinear rate when $f$ is bounded below
\cite{Nes04}. (The classical theory for gradient descent says little
about the case in which $f$ is unbounded below, as in this example.)

The gradient descent and heavy-ball methods will converge to the
saddle point $0$ for \eqref{eq:q} only from starting points of the
form $x^0=(x^0_1,0)$ for any $x^0_1 \in \R$.  (In the case of
heavy-ball, this claim follows from Theorem~\ref{th:x0x0}, using the
fact that $(1,0)^T$ is the eigenvector of $\nabla^2 f$ that
corresponds to the positive eigenvalue $1$.)  From any other starting
point, both methods will diverge, with function values going to
$-\infty$.  When the starting point $x^0$ is very close to (but not
on) the $x_1$ axis, the typical behavior is that these algorithms pass
close to $0$ before diverging along the $x_2$ axis.  We are interested
in the question: {\em Does the heavy-ball method diverge
  away from $0$ significantly faster than the steepest-descent
  method?} The answer is ``yes,'' as we show in this section.
  
  We consider a starting point that is just off the horizontal axis,
that is,
\begin{equation} \label{eq:x0}
x^0 = \left[ \begin{matrix} 1        \\ \eps \end{matrix} \right], \quad \mbox{for some small $\epsilon>0$.}
\end{equation}
For the steepest-descent method with constant steplength, we have
\[
\left[ \begin{matrix} x^{k+1}_1      \\ x^{k+1}_2 \end{matrix} \right] =
\left[ \begin{matrix} x^{k}_1        \\ x^{k}_2 \end{matrix} \right] - 
\alpha \left[ \begin{matrix} x^{k}_1 \\ -\delta x^{k}_2 \end{matrix} \right],
\]
so that 
\begin{equation} \label{eq:sd.q}
\left[ \begin{matrix} x^{k}_1        \\ x^{k}_2 \end{matrix} \right] =
\left[ \begin{matrix} (1-\alpha)^k   \\ (1+\delta \alpha)^k \epsilon \end{matrix} \right].
\end{equation}
One measure of repulsion from the saddle point is the number of
iterations required to obtain $|x^k_2| \ge 1$. Here it suffices for
$k$ to be large enough that $(1+\delta \alpha)^k \epsilon \ge 1$, for
which (using the usual  bound $\log(1+\gamma) \le \gamma$) a
sufficient condition is that
\[
k \ge \frac{|\log \epsilon|}{\delta \alpha}.
\]
Making the standard choice of steplength $\alpha = 1/L = 1$, we obtain
\begin{equation} \label{eq:sd.c1}
k \ge \frac{|\log \epsilon|}{\delta}.
\end{equation}

Consider now the heavy-ball method. Following \eqref{eq:hb}, the
iteration has the form:
\begin{equation} \label{eq:hb.q}
\left[ \begin{matrix} x^{k+1}_1                                                                      \\ x^{k+1}_2 \end{matrix} \right] =
\left[ \begin{matrix} (1-\alpha) x^{k}_1                                                             \\ (1+\delta \alpha) x^{k}_2 \end{matrix} \right] + \beta 
\left[ \begin{matrix} x^{k}_1-x^{k-1}_1                                                              \\ x^{k}_2 -x^{k-1}_2 \end{matrix} \right].
\end{equation}
(For this quadratic problem, the operator $G$ defined by
\eqref{eq:hb.map} is linear, so that $DG$ is constant.)  We can
partition this recursion into $x_1$ and $x_2$ components, and write
\begin{equation} \label{eq:toy.M.1}
\left[ \begin{matrix} x^{k+1}_1 \\ x^k_1  \end{matrix} \right]
=
M_1 
\left[ \begin{matrix} x^{k}_1 \\ x^{k-1}_1  \end{matrix} \right], \quad
\left[ \begin{matrix} x^{k+1}_2 \\ x^k_2  \end{matrix} \right]
=
M_2
\left[ \begin{matrix} x^{k}_2 \\ x^{k-1}_2  \end{matrix} \right], 
\end{equation}
where
\begin{equation} \label{eq:toy.M.2}
M_1 = 
\left[ \begin{matrix} 1-\alpha+\beta & -\beta \\1 & 0  \end{matrix} \right], \quad
M_2 = 
\left[ \begin{matrix} 1+\delta\alpha+\beta & -\beta \\1 & 0  \end{matrix} \right].
\end{equation}
The eigenvalues of these two matrices are given by \eqref{eq:roots},
by setting $\lambda_1=1$ and $\lambda_2=-\delta$, respectively.
For $\alpha$ and $\beta$ satisfying the conditions of
Theorem~\ref{th:DG}, which translate here to
\begin{equation} \label{eq:toy.ab}
0 < \alpha < 4, \quad \beta \in (-1+\alpha/2,1),
\end{equation}
both eigenvalues of $M_1$ are less than $1$ in magnitude (as we show
in the proof of Theorem~\ref{th:DG}), so the $x_1$ components converge
to zero. Again referring to the proof of Theorem~\ref{th:DG}, the
eigenvalues of $M_2$ are both real, with one of them greater than $1$,
suggesting divergence in the $x_2$ component.

To understand rigorously the behavior of the $x_2$ sequence, we make
some specific choices of $\alpha$ and $\beta$. Consider
\begin{equation} \label{eq:toy.ab2}
\alpha \in (0,3], \quad \beta = 1-\alpha \delta - \gamma,
\end{equation}
for some parameter $\gamma \geq 0$. Note that for small $\delta$ and
$\gamma$, these choices are consistent with \eqref{eq:toy.ab}.  By
substituting into \eqref{eq:roots}, we see that the two eigenvalues of
$M_2$ are
\[
\mu_2^{\text{hi,lo}} = \frac12 \left[ (2-\gamma) \pm \sqrt{\gamma^2 + 4\alpha
    \delta} \right].
\]
For reasonable choices of $\gamma$, we have that $\mu_2^{\text{hi}} = 1 + c
\sqrt{\delta}$ for a modest positive value of $c$. For specificity
(and simplicity) let us consider $\alpha=3$ and $\gamma=0$, for which
we have
\begin{equation} \label{eq:toy.mu2}
\mu_2^{\text{hi}} = 1+\sqrt{3\delta}, \quad \mu_2^{\text{lo}} = 1-\sqrt{3 \delta}.
\end{equation}
The formula \eqref{eq:Mi.S} yields $M_2 = S_2 \Lambda_2 S_2^{-1}$,
where $\Lambda_2 = \diag( 1+\sqrt{3 \delta} , 1-\sqrt{3 \delta})$ and
\[
S_2 = 
\left[ \begin{matrix} 1+\sqrt{3 \delta} & 1 \\ 1 & \frac{1}{1-\sqrt{3\delta}} \end{matrix} \right], \quad
S_2^{-1} = \frac{1-\sqrt{3 \delta}}{2 \sqrt{3 \delta}} 
\left[ \begin{matrix} \frac{1}{1-\sqrt{3 \delta}} & -1 \\
-1 & 1 + \sqrt{3\delta} \end{matrix} \right].
\]
From \eqref{eq:toy.M.1}, and setting $x^0_2=x^{-1}_2 = \epsilon$, we have
\[
\left[ \begin{matrix} x^{k}_2 \\ x^{k-1}_2 \end{matrix} \right] = 
S_2 \Lambda_2^k S_2^{-1} 
\left[ \begin{matrix} \epsilon \\ \epsilon \end{matrix} \right].
\]
By substituting for $\Lambda_2$ and $S_2$, we obtain
\begin{align*}
\left[ \begin{matrix} x^{k}_2 \\ x^{k-1}_2 \end{matrix} \right] & = \epsilon
S_2 \Lambda_2^k S_2^{-1} 
\left[ \begin{matrix} 1 \\ 1 \end{matrix} \right] \\
&= \epsilon \frac{1-\sqrt{3 \delta}}{2 \sqrt{3 \delta}}  S_2 \Lambda_2^k
\left[ \begin{matrix} \frac{\sqrt{3 \delta}}{1-\sqrt{3 \delta}} \\
\sqrt{3 \delta} \end{matrix} \right] \\
&= \epsilon S_2 \Lambda_2^k 
\left[ \begin{matrix} 1/2 \\ (1-\sqrt{3 \delta})/2 \end{matrix} \right] \\
&= \epsilon S_2 
\left[ \begin{matrix} (\mu_2^{\text{hi}})^k /2 \\
(\mu_2^{\text{lo}})^k (1-\sqrt{3 \delta})/2 \end{matrix} \right] \\
& \ge \frac12 \epsilon  
\left[ \begin{matrix} (1+\sqrt{3 \delta}) (\mu_2^{\text{hi}})^k  \\
(\mu_2^{\text{hi}})^k  \end{matrix} \right],
\end{align*}
where we simply drop the term involving $\mu_2^{\text{lo}}$ in the final step
and use $1-\sqrt{3 \delta}>0$. It follows that
\[
x_2^k \ge \frac12 \epsilon (1+\sqrt{ 3 \delta}) (\mu_2^{\text{hi}})^k = \frac12
\epsilon (1+\sqrt{3 \delta})^{k+1}.
\]
It follows from this bound, by a standard argument, that a sufficient
condition for $x_2^k \ge 1$ is
\[
k+1 \ge \frac{\log(2/\epsilon)}{\sqrt{3\delta}}.
\]

Thus we have confirmed that divergence from the saddle point occurs in \\
$O(| \log \epsilon|/\sqrt{\ddd})$ iterations for heavy-ball, versus
$O(| \log \epsilon|/\ddd )$ iterations for gradient descent.

For larger values of $\ddd$, the divergence of steepest-descent and
heavy-ball methods are both rapid, For appropriate choices of $\alpha$
and $\beta$, the iterates generated by both algorithms leave the
vicinity of the saddle point quickly.

\begin{figure}
\centering\includegraphics[width=4.5in]{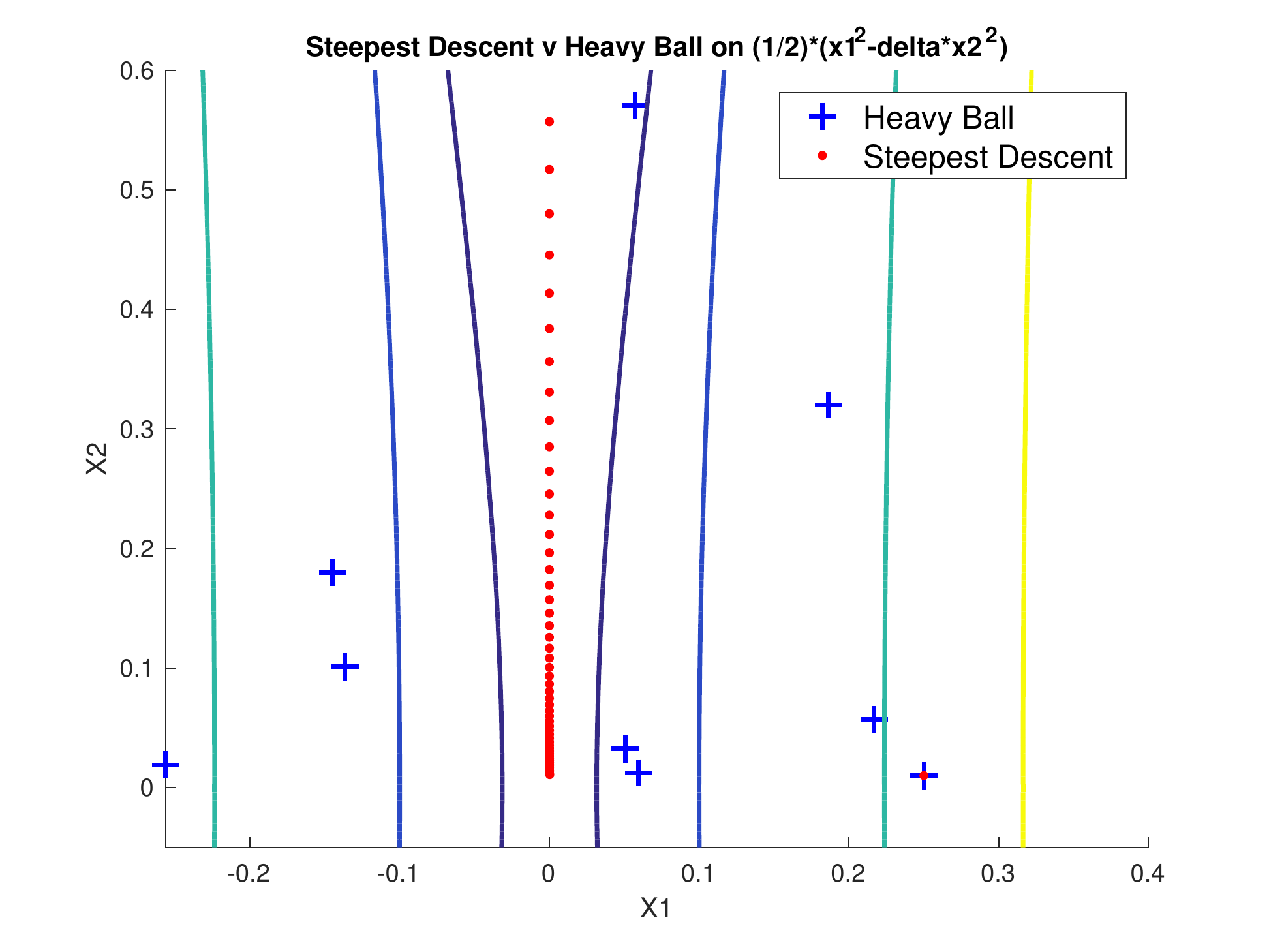}
\caption{Steepest descent and heavy-ball on \eqref{eq:q}
  with $\delta=.02$, from starting point $(.25,.01)^T$, with
  $\alpha=.75$, $\beta=1-\alpha \delta = .985$. Every
  $5$th iterate is plotted for each method.\label{fig:escape1} }
\end{figure}

Figure~\ref{fig:escape1} illustrates the divergence behavior of
steepest descent and heavy-ball on the function \eqref{eq:q} with
$\ddd=.02$. We set $\alpha=.75$ for both steepest descent and
heavy-ball. For heavy-ball, we chose $\beta=1-\alpha \delta = .985$.
Both methods were started from $x^0=(.25,.01)^T$. We see that the
trajectory traced by steepest descent approaches the saddle point
quite closely before diverging slowly along the $x_2$ axis. The
heavy-ball method ``overshoots'' the $x_2$ axis (because of the
momentum term) but quickly returns to diverging along the $x_2$
direction at a faster rate than for steepest descent.

\section{General Accelerated Gradient Methods Applied to Quadratic Functions}
\label{sec:agd}

Here we analyze the rate at which a general class of accelerated gradient methods
escape the saddle point of an $n$-dimensional quadratic function:
\begin{equation}
\underset{x \in \mathbb{R}^n} {\min} f(x) = \frac{1}{2} x^T H x \label{eq:quadf}
\end{equation}
where $H$ is a symmetric matrix with eigenvalues satisfying \eqref{eq:np}.  We assume
without loss of generality that $H$ is in fact diagonal, that is,
\begin{equation} \label{eq:Hdiag}
H = \diag (\lambda_1, \lambda_2, \dotsc, \lambda_n).
\end{equation}
The Lipschitz constant $L$ for $\nabla f$ is $L=\max(\lambda_1,-\lambda_n)$.

\begin{algorithm}[ht!]
\caption{General Accelerated Gradient Framework}
\label{alg:gradmethods}
\begin{algorithmic}
\STATE Choose $x^1 \in \R^n$, $\alpha < \frac{1}{L}$;
\STATE Set $x^0=x^1$;
\FOR{$k=1,2,\dotsc$}
\STATE Choose $\gamma_k \in [0,1]$ and $\beta_k \in [0,1]$;
\STATE $y^k = x^k + \gamma_k (x^k - x^{k-1})$;
\STATE $x^{k+1} = x^k + \beta_k (x^k - x^{k-1}) - \alpha \nabla f(y^k)$;
\ENDFOR
\end{algorithmic}
\end{algorithm}

As in Section~\ref{sec:simplest},
gradient descent with $\alpha \in (0,1/L)$ satisfies
\begin{equation}
x^{k+1}_i =  (1-\alpha \lambda_i) x^k_i  = (1-\alpha \lambda_i)^k x^1_i,
\quad i=1,2,\dotsc,n. \label{eq:gdupdate}
\end{equation}
It follows that for all $i \geq n-p+1$, for which $\lambda_i<0$,
gradient descent diverges in that component at a rate of $(1-\alpha
\lambda_i)$.

Algorithm~\ref{alg:gradmethods} describes a general accelerated gradient framework,
including gradient descent when $\gamma_k = \beta_k = 0$, heavy-ball when
$\gamma_k = 0$ and $\beta_k > 0$, and accelerated gradient methods
when $\gamma_k = \beta_k > 0$. With $f$ defined by \eqref{eq:quadf},
the update formula can be written as
\begin{flalign*}
x^{k+1} &= x^k + \beta_k(x^k - x^{k-1}) - \alpha H(x^k + \gamma_k(x^k - x^{k-1})) \\
&= ((1 + \beta_k) I - \alpha (1 + \gamma_k) H)x^k - (\beta_k I - \alpha \gamma_k H) x^{k-1},
\end{flalign*}
which because of \eqref{eq:Hdiag} is equivalent to
\begin{equation}
x^{k+1}_i =  ((1 + \beta_k) -(1 + \gamma_k)  \alpha \lambda_i)x_i^k -
(\beta_k - \gamma_k  \alpha \lambda_i) x_i^{k-1}, \quad i=1,2,\dotsc,n. \label{eq:xupdate}
\end{equation}

%%%%
%It is easy to see from the definitions of $t_k$ and $\gamma_k$ that
%\begin{subequations}
%\label{eq:tg1}
%\begin{align}
%& \{ t_k \}_{k=0,1,2,\dotsc} \;\; \mbox{ is an increasing sequence, with $t_0=1$;} \\
%& \{\gamma_k\}_{k=1,2,\dotsc,} \;\; \mbox{ is an increasing sequence, with $\gamma_1=0$.}
%\end{align}
%\end{subequations}
%It can also be shown that
%\begin{equation} \label{eq:gt2}
%\frac{\gamma_k}{t_k} \;\; \mbox{is an increasing nonnegative sequence}
%\end{equation}
%and
%\begin{equation} \label{eq:gt3}
%\lim_{k \rightarrow \infty} \frac{\gamma_k}{t_k} = 1.
%\end{equation}
%(For completeness, we include proofs of these two claims in the appendix.)

The following theorem describes the dynamics of $x^{k+1}_i$ in
\eqref{eq:xupdate} when $\lambda_i < 0$.
\begin{theorem}
\label{thm:recursiveb}
For all $i$ such that $\lambda_i < 0$, we have from \eqref{eq:xupdate}
that
\begin{equation}
x_i^{k+1} = x_i^0 \prod_{m=0}^k (1 + b_{i,m})  \label{eq:xprod}
\end{equation}
where
\begin{equation}
b_{i,k} =
\begin{cases}
0, & \mbox{for } k=0 \\
(\beta_k + \gamma_k \alpha |\lambda_i|)\left(1 - \frac{1}{1 + b_{i,k-1}}\right) + \alpha |\lambda_i|,
& \mbox{otherwise.}
\end{cases}
\label{eq:bdef}
\end{equation}
In addition if $\gamma_{k+1} \geq \gamma_k$ and $\beta_{k+1} \geq \beta_k$ for all $k$ then,
\begin{equation}
b_{i,k+1} \geq b_{i,k}, \quad k=1,2,\dotsc.
\label{eq:bleq1}
\end{equation}
\end{theorem}

\begin{proof}
We begin by showing that (\ref{eq:bdef}) holds for $k=0$ and $k=1$.
The case for $k=0$ is trivial as $x^1 = x^0$.
In addition, for $k=1$, the update formula \eqref{eq:xupdate} becomes
\[
x_i^{2} = (1-\alpha \lambda_i)x_i^0.
\]
Thus because $b_{i,0} = 0$, we can make this consistent with (\ref{eq:xprod}) by setting
$b_{i,1} = \alpha |\lambda_i|$ which is exactly (\ref{eq:bdef}) for $k = 1$.

Now assume that (\ref{eq:xprod}) holds for all $k \leq K-1$.  From (\ref{eq:xupdate}),
using the inductive hypothesis for $K-1$ and $K-2$, we need to show
\begin{multline}
x^0_i \prod_{m=0}^K (1+b_{i,m}) = x_i^0((1 + \beta_K) -
(1 + \gamma_K) \alpha \lambda_i) \prod_{m=0}^{K-1} (1 + b_{i,m}) \\
- x_i^0 (\beta_K - \gamma_K \alpha \lambda_i) \prod_{m=0}^{K-2} (1 + b_{i,m})
 \label{eq:thm1eq1}
\end{multline}
by the given definition of $b_{i,K}$ in (\ref{eq:bdef}).
Dividing both sides by $x^0_i \prod_{m=0}^{K-1} (1 + b_{i,m})$, this is equivalent to
\[
1 + b_{i,K} = 1 + \beta_K + (1 + \gamma_K) \alpha |\lambda_i| -
\frac{\beta_K + \gamma_K \alpha |\lambda_i|}{1 + b_{i,K-1}},
\]
which is true because
\[
b_{i,K} = (\beta_K + \gamma_K \alpha |\lambda_i|)\left(1 - \frac{1}{1 + b_{i,K-1}}\right)
+ \alpha |\lambda_i|
\]
is (\ref{eq:bdef}) with $k = K$, as required.

Now we assume that $\gamma_{K+1} \geq \gamma_K$ and $\beta_{K+1} \geq \beta_K$
holds for all $K \geq 1$ and show by induction that $b_{i,K+1} \geq b_{i,K}$ holds for all $K \ge 0$.
This is clearly true for $K=0$ since $\alpha |\lambda_i| > 0$.  Assume now
that $b_{i,k+1} \geq b_{i,k}$ holds for all $0 \leq k \leq K-1$.
We have
\begin{alignat*}{2}
b_{i,K+1} &= (\beta_{K+1} + \gamma_{K+1} \alpha |\lambda_i|)\left(1 - \frac{1}{1 + b_{i,K}}\right)
+ \alpha |\lambda_i| \\
&\geq (\beta_{K+1} + \gamma_{K+1} \alpha |\lambda_i|)\left(1 - \frac{1}{1 + b_{i,K-1}}\right)
+ \alpha |\lambda_i| \\
&\geq (\beta_{K} + \gamma_{K} \alpha |\lambda_i|)\left(1 - \frac{1}{1 + b_{i,K-1}}\right)
+ \alpha |\lambda_i| = b_{i,K}.
\end{alignat*}
where the second inequality above follows from $\gamma_{K+1} \geq \gamma_K$, 
$\beta_{K+1} \geq \beta_K$ and $b_{i,K-1} \ge b_{i,0} = 0$.
\end{proof}

Since $b_{i,k} \geq \alpha |\lambda_i|$ for all $k \geq 1$,
Theorem~\ref{thm:recursiveb} shows that Algorithm~\ref{alg:gradmethods}
diverges at a faster rate than gradient descent when at least one of
$\gamma_k > 0$ or $\beta_k > 0$ is true.  Now we explore the rate
of divergence by finding a limit for the sequence $\{b_{i,k}\}_{k=1,2,\dotsc}$.

\begin{theorem}
Let $\gamma_{k+1} \geq \gamma_k$ and $\beta_{k+1} \geq \beta_k$ hold for all $k$ and
denote $\bar{\gamma} = \lim_{k \to \infty} \gamma_k$ and
$\bar{\beta} = \lim_{k \to \infty} \beta_k$.
Then, for all $i$ such that $\lambda_i < 0$, we have \\
$\lim_{k \rightarrow \infty} b_{i,k} = \bar{b}_i$, where $\bar{b}_i$ is defined by by
\begin{equation} \label{eq:bbdef}
\bar{b}_i := \frac12 \left(\bar{\beta} - 1 + \alpha |\lambda_i|(1 + \bar{\gamma}) \right)
+ \frac12 \sqrt{(\bar{\beta} - 1 + \alpha |\lambda_i|(1 + \bar{\gamma}))^2 + 4\alpha |\lambda_i|}
\end{equation}
\label{thm:bbarbound}
\end{theorem}
\begin{proof}
We can write (\ref{eq:xupdate}) as follows:
\[
x^{k+1}_i = (1 + \alpha |\lambda_i|) x^k_i +
(\beta_k + \gamma_k \alpha |\lambda_i|) (x^k_i - x^{k-1}_i). 
\]
Recall from Theorem \ref{thm:recursiveb} that $x_i^k = (1 + b_{i,k-1})
x_i^{k-1}$.  By substituting into the equation above, we have
\begin{flalign}
x^{k+1}_i &= \left[(1 + \alpha |\lambda_i|) (1 + b_{i,k-1}) + (\beta_k + \gamma_k \alpha |\lambda_i|) 
b_{i,k-1} \right] x^{k-1}_i \nonumber \\
&= \left[1 + \alpha |\lambda_i| + \left(1 + \alpha |\lambda_i| + \beta_k +
\gamma_k \alpha |\lambda_i| \right) b_{i,k-1} \right] x^{k-1}_i. \label{eq:bbarthmeq1}
\end{flalign}
Using Theorem~\ref{thm:recursiveb} again, we have
\[
x^{k+1}_i = \left[ (1 + b_{i,k}) (1 + b_{i,k-1})\right] x_i^{k-1} 
= \left[ 1 + b_{i,k} + b_{i,k-1} + b_{i,k-1}b_{i,k} \right] x_i^{k-1}. 
\]
By matching this expression with (\ref{eq:bbarthmeq1}), we obtain
\begin{equation}
\alpha |\lambda_i| + \left(1 + \alpha |\lambda_i| + \beta_k +
\gamma_k \alpha |\lambda_i| \right) b_{i,k-1}
=  b_{i,k} + b_{i,k-1} + b_{i,k-1}b_{i,k}, \label{eq:bbarthmeq2}
\end{equation}
which after division by $b_{i,k-1}$ yields
\begin{equation} \label{eq:bbardiv}
\frac{\alpha |\lambda_i|}{b_{i,k-1}} + \left(1 + \alpha |\lambda_i| + \beta_k +
\gamma_k \alpha |\lambda_i| \right) =  \frac{b_{i,k}}{b_{i,k-1}} + 1 + b_{i,k}.
\end{equation}

Now assume for contradiction that the nondecreasing sequence $\{
b_{i,k}\}_{k=1,2,\dotsc}$ has no finite limit, that is, $b_{i,k}
\rightarrow \infty$. Recalling that $\gamma_k$ and $\beta_k$ have a finite
limit (as they are nondecreaseing sequences restricted to the interval $[0,1]$),
we have by taking the limit as
$k \to \infty$ in (\ref{eq:bbardiv}) that the left-hand side approaches $\left(1 +
\alpha |\lambda_i| + \bar{\beta} + \bar{\gamma} \alpha |\lambda_i| \right)$,
while the right-hand side approaches $\infty$, a contradiction. Thus, the
nondecreasing sequence $\{b_{i,k}\}_{k=1,2,\dotsc}$ has a finite limit,
which we denote by $\bar{b}_i$.

To find the value for $\bar{b}_i$, we take limits as $k \rightarrow
\infty$ in (\ref{eq:bbarthmeq2}) to obtain
\[
\alpha |\lambda_i| + \left(1 + \alpha |\lambda_i| + \bar{\beta} + \bar{\gamma} \alpha |\lambda_i|
\right) \bar{b}_i =  2 \bar{b}_i + \bar{b}_i^2.
\]
%% By rearranging this equation, we obtain
%% \[
%% \bar{b}_i^2 - 2 \alpha |\lambda_i| \bar{b}_i - \alpha |\lambda_i| = 0.
%% \]
By solving this quadratic for $\bar{b}_i$, we obtain
%% \begin{align*}
%% \bar{b}_i &= \frac{1}{2}\left[ 2 \alpha |\lambda_i| \pm \sqrt{\left(2 \alpha|\lambda_i|\right)^2 + 4 \alpha |\lambda_i|} \right] \\
%% &= \alpha |\lambda_i| \pm \sqrt{\left(\alpha|\lambda_i|\right)^2 + \alpha |\lambda_i|} \\
%% &= \alpha |\lambda_i| \pm \sqrt{\alpha |\lambda_i|} \sqrt{1 + \alpha |\lambda_i|}
%% \end{align*}
\[
\bar{b}_i = \frac12 \left(\bar{\beta} - 1 + \alpha |\lambda_i|(1 + \bar{\gamma}) \right)
\pm \frac12 \sqrt{(\bar{\beta} - 1 + \alpha |\lambda_i|(1 + \bar{\gamma}))^2 + 4\alpha |\lambda_i|}.
\]
By Theorem~\ref{thm:recursiveb}, we know that $b_{i,k} \geq 0$ for all $k$,
so that $\bar{b}_i \geq 0$. Therefore, $\bar{b}_i$ satisfies \eqref{eq:bbdef}, as claimed.
\end{proof}

We apply Theorem \ref{thm:bbarbound} to parameter choices that typically appear in
accelerated gradient methods.

\begin{corollary} \label{cor:agdlim}
Let the assumptions of Theorem~\ref{thm:bbarbound} hold, let $\gamma_k = \beta_k$ hold
for all $k$ and let $\bar{\gamma} = \bar{\beta} = 1$. Then,
\begin{equation} \label{eq:accellim}
\bar{b}_i = \alpha |\lambda_i| + \sqrt{\alpha |\lambda_i|} \sqrt{1 + \alpha |\lambda_i|}.
\end{equation}
\end{corollary}

\begin{proof}
By direct computation with $\bar{\beta} = \bar{\gamma} = 1$, we have
\[
\bar{b}_i = \alpha |\lambda_i| + \frac12 \sqrt{4 (\alpha |\lambda_i|)^2 + 4 \alpha |\lambda_i|}
= \alpha |\lambda_i| + \sqrt{\alpha |\lambda_i|} \sqrt{1 + \alpha |\lambda_i|}.
\]
\end{proof}

The above corollary gives a rate of divergence for many standard choices of the
extrapolation parameters found in the accelerated gradient literature. In particular,
it includes the sequence $\beta_k = \gamma_k = \frac{t_{k-1}-1}{t_k}$ where $t_0 = 1$
and
\begin{equation} \label{eq:nesseq}
t_{k} = \frac{\sqrt{4 t_{k-1}^2 + 1} + 1}{2}
\end{equation}
which was used in a seminal work by Nesterov~\cite{Nes83}.
(For completeness, we provide a proof that  $t_k \to \infty$, so that
the assumptions of Corollary \ref{cor:agdlim} hold
for this sequence in the appendix.) Another setting used in recent works
$\beta_k = \gamma_k = \frac{k-1}{k+\eta+1}$ \cite{attouch2018convergence} 
\cite{attouch2016rate} \cite{chambolle2015convergence}. For proper choices of $\eta > 0$,
this scheme has a number of impressive properties such as fast convergence of iterates
for accelerated proximal gradient as well as achieving a $o(\frac{1}{k^2})$
of convergence in the weakly convex case.

We can also use Theorem~\ref{thm:bbarbound} to derive a bound for the
heavy-ball method. If we target the $n$-th eigenvalue and set
$\gamma_k = 0$ and $\beta_k = 1 - \alpha |\lambda_n|$ for all $k$,
simple manipulation shows that $\bar{b}_n = \sqrt{\alpha
  |\lambda_n|}$, which gives us an equivalent rate to that derived in
(\ref{eq:toy.mu2}). Note that for $\bar{b}_n$ defined in
\eqref{eq:accellim} we also have $\bar{b}_n \ge \sqrt{\alpha |
  \lambda_n|}$.

The divergence rates for accelerated gradient and heavy-ball methods
are significantly faster than the per-iteration rate of $(1+\alpha
|\lambda_n|)$ obtained for steepest descent.

%We have shown that the limiting rate of divergence for Nesterov's
%method for the negative-eigenvalue components is $1 + \alpha
%|\lambda_i| + \sqrt{1 + \alpha|\lambda_i|} \sqrt{\alpha |\lambda_i|}$,
%a significant improvement on the rate of $1 + \alpha |\lambda_i|$
%attained by gradient descent. For the toy problem of the previous
%section, with $\lambda_1=1$, $\lambda_2=-\delta$, and $\alpha=1$, we
%would see similar divergence behavior in the $x_2$ component for
%Nesterov's method as for the heavy-ball method --- $1+c \sqrt{\delta}$
%in both cases, for a modest positive value of $c$.

\section{Experiments} \label{sec:experiments}

Some computational experiments verify that accelerated gradient
methods escape saddle points on nonconvex quadratics faster than
steepest descent.

We apply these methods to a quadratic with diagonal Hessian, with
$n=100$ and a single negative eigenvalue, $\lambda_n = -\delta = -0.01$.
The nonnegative eigenvalues are i.i.d. from the uniform distribution on
$[0,1]$, and starting vector $x^0$ is drawn from a uniform
distribution on the unit ball. Figure~\ref{fig:negspaceplot} plots the
norm of the component of $x^k$ in the direction of the negative
eigenvector $e_n = (0,0, \dotsc,0,1)^T$ at each iteration $k$, for
accelerated gradient, heavy-ball, and steepest descent. It also shows the
divergence that would be attained if the theoretical limit $\bar{b}_i$
from Theorem \ref{thm:bbarbound} applied at every iteration. Steepest
descent and heavy-ball were run with $\alpha = 1/L$. Heavy-ball uses
\eqref{eq:toy.ab2} to calculate $\beta$, yielding
$\beta = 0.989$ in the case of $\delta=.01$.
Accelerated gradient is run with $\alpha = 0.99/L$ and
$\beta_k = \gamma_k = \frac{t_k-1}{t_{k+1}}$ where $t_k$ is defined in (\ref{eq:nesseq}).

It is clear from Figure~\ref{fig:negspaceplot} that accelerated gradient
and heavy-ball diverge at a significantly faster rate than steepest
descent.  In addition, there is only a small discrepancy between
applying accelerated gradient and its limiting rate that is derived in
Corollary \ref{cor:agdlim}, suggesting that $b_{i,k}$ approaches
$\bar{b}_i$ rapidly as $k \to \infty$.

Next we investigate how these methods behave for various dimensions
$n$ and various distributions of the eigenvalues. For two values of
$n$ ($n=100$ and $n=1000$), we generate $100$ random matrices with
$n-5$ eigenvalues uniformly distributed in the interval $[0,1]$, with
the $5$ negative eigenvalues uniformly distributed in $[-2\delta,
  -\delta]$. The starting vector $x^0$ is uniformly distributed on the
unit ball. Algorithmic constants were the same as those used to
generate Figure \ref{fig:negspaceplot}. Each trial was run until the 
norm of the projection of the current iterate into the negative eigenspace
of the Hessian was greater than the dimension $n$. The results of these trials
are shown in Table~\ref{tab:n100}.

As expected, accelerated gradient outperforms gradient descent in all
respects. All convergence results are slightly faster for $n=100$ than
for $n=1000$, because the random choice of $x^0$ will, in expectation,
have a smaller component in the span of the negative eigenvectors in
the latter case. The eigenvalue spectrum has a much stronger effect on
the divergence rate. For steepest descent, an order of magnitude
decrease in the absolute value of the negative eigenvalues corresponds
to an order of magnitude increase in iterations, whereas Nesterov's
accelerated gradient sees significantly less growth in the iteration
count. While the accelerated gradient method diverges at a slightly
slower rate than the theoretical limit, the relative difference between the two does
not change much as the dimensions change. Thus,
Theorem~\ref{thm:bbarbound} provides a strong indication of the
practical behavior of Nesterov's method on these problems.

\begin{figure}[h]
\caption{Momentum methods and theoretical divergence applied to a
  quadratic function with $n=100$ and a single negative eigenvalue.
  The vertical axis displays the norm of the projection of $x^k$ onto
  the negative eigenvector.}
\label{fig:negspaceplot}
\centering
\includegraphics[width=\textwidth]{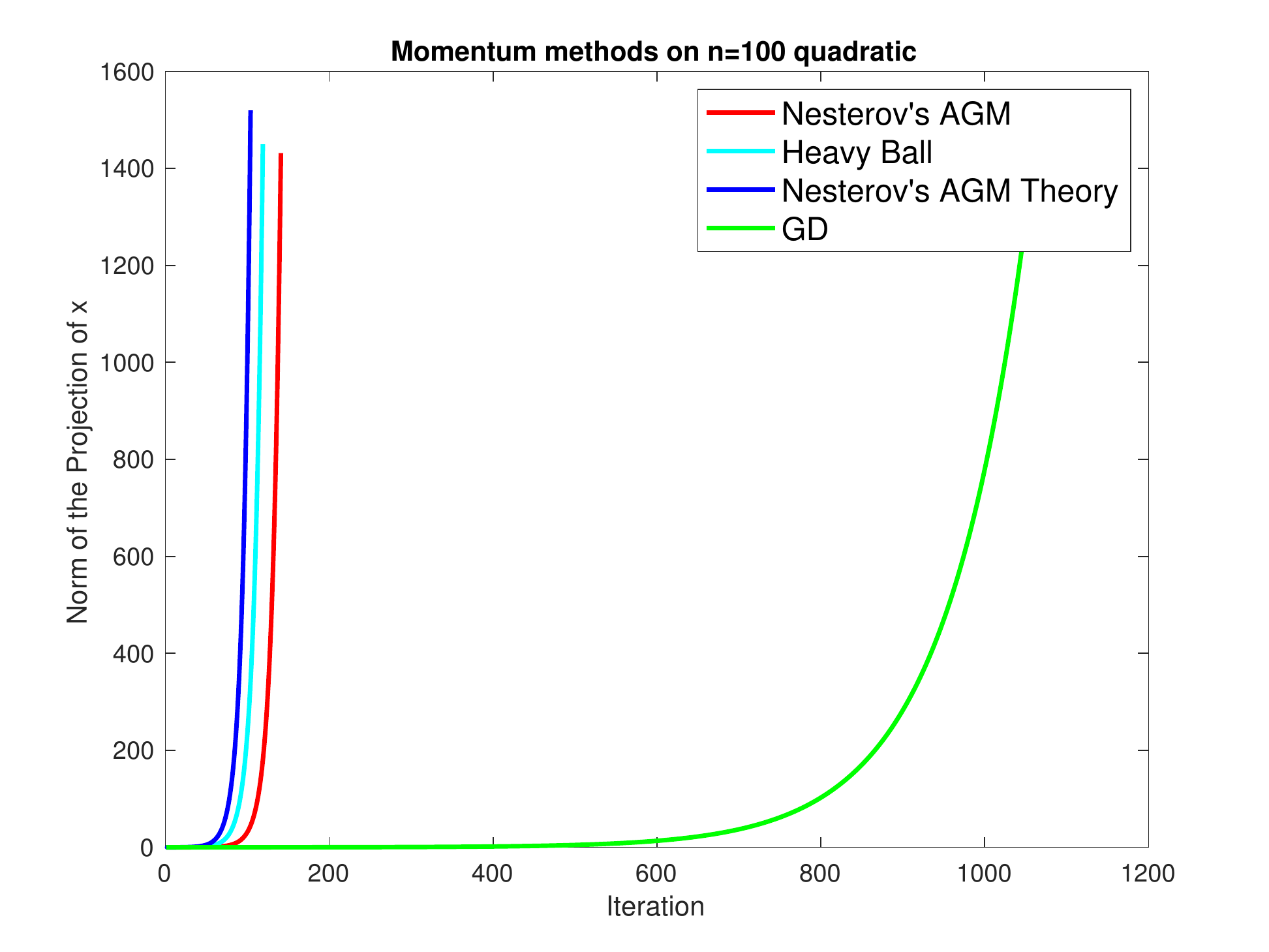}
\end{figure}

\begin{table}
\small
\caption{Divergence Behavior of Gradient Algorithms}
\begin{center}
\begin{tabular}{ |r|r|c|r|r|c| } 
 \hline
$n$ & $\delta$ & Method & Av. Iters & Max. Iters \\
 \hline
  && Steepest Descent & 379 & 518 \\
  $100$ & $10^{-2}$ & Accelerated Gradient & 71 & 87 \\
  && $\bar{b}$ Divergence Rate & 46 & 59 \\
 \hline
  && Steepest Descent & 3855 & 5603 \\
  $100$ & $10^{-3}$ & Accelerated Gradient & 242 & 299 \\ 
  && $\bar{b}$ Divergence Rate  & 155 & 194  \\
 \hline
  && Steepest Descent & 582 & 773 \\
  $1000$ & $10^{-2}$ & Accelerated Gradient & 99 & 116 \\
  && $\bar{b}$ Divergence Rate & 71 & 85 \\
 \hline
  && Steepest Descent & 5775 & 8240 \\
  $1000$ & $10^{-3}$ & Accelerated Gradient & 332 & 399 \\ 
  && $\bar{b}$ Divergence Rate & 235 & 282  \\
 \hline
\end{tabular}
\end{center}
\label{tab:n100}
\end{table}

%% \begin{table}
%% \small
%% \caption{Algorithm Divergence for $n=1000$}
%% \begin{center}
%% \begin{tabular}{ |c|c|c|c|c| } 
%%  \hline
%%  $\delta$ & Method & Average Iterations & Max Iterations & Average Time \\
%%  \hline
%%   & GD & 582 & 773 & 0.36 \\
%%   $10^{-2}$ & Nesterov's AGM & 99 & 116 & 0.17 \\
%%   & $\bar{b}$ Divergence Rate & 71 & 85 & N/A \\
%%  \hline
%%   & GD & 5775 & 8240 & 3.85 \\
%%   $10^{-3}$ & Nesterov's AGM & 332 & 399 & 0.61 \\ 
%%   & $\bar{b}$ Divergence Rate & 235 & 282 & N/A \\
%%  \hline
%% \end{tabular}
%% \end{center}
%% \label{tab:n1000}
%% \end{table}

\section{Conclusion}

We have derived several results about the behavior of accelerated
gradient methods on nonconvex problems, in the vicinity of critical
points at which at least one of the eigenvalues of the Hessian
$\nabla^2 f(x^*)$ is negative.  Section~\ref{sec:sm} shows that the
heavy-ball method does not converge to such a point when started
randomly, while Sections~\ref{sec:simplest} and \ref{sec:agd} show
that when $f$ is an indefinite quadratic, momentum methods diverge
faster than the steepest-descent method.

It would be interesting to extend the results on speed of divergence
to non-quadratic smooth functions $f$. It would also be interesting to
know what can be proved about the complexity of convergence to a point
satisfying second-order necessary conditions, for unadorned
accelerated gradient methods. A recent work \cite{du2017gradient}
shows that gradient descent can take exponential time to escape from a
set of saddle points.  We believe that a similar result holds for
accelerated methods as well. The report \cite{jin2017accelerated},
which appeared after this paper was submitted, describes an
accelerated gradient method that add noise selectively to some
iterates, and exploits negative curvature search directions when they
are detected in the course of the algorithm. This approach is shown to
have the $O(\epsilon^{-7/4})$ rate that characterizes the best known
gradient-based algorithms for finding second-order necessary points of
smooth nonconvex functions.

\section*{Acknowledgments} 
We are grateful to Bin Hu for his advice and suggestions on the
manuscript. We are also grateful to the referees and editor for
helpful suggestions.

\bibliographystyle{plain} \bibliography{refs}

\appendix

\section{Properties of the  Sequence $\{t_k\}$ Defined By (\ref{eq:nesseq})}

In this appendix we show that the following two properties hold for
the sequence defined by \eqref{eq:nesseq}:
\begin{equation} \label{eq:gt2}
\frac{t_{k-1}-1}{t_k} \;\; \mbox{is an increasing nonnegative sequence}
\end{equation}
and
\begin{equation} \label{eq:gt3}
\lim_{k \rightarrow \infty} \frac{t_{k-1}-1}{t_k} = 1.
\end{equation}
We begin by noting two well known properties of the sequence $t_k$
(see for example  \cite[Section~3.7.2]{bubeck2015convex}):
\begin{equation} \label{eq:tsquared}
t_k^2 - t_k = t_{k-1}^2
\end{equation}
and
\begin{equation} \label{eq:tlb}
t_k \geq \frac{k+1}{2}.
\end{equation}
To prove that $\frac{t_{k-1}-1}{t_k}$ is monotonically increasing, we
need
\[
\frac{t_{k-1}-1}{t_k} = \frac{t_{k-1}}{t_k} - \frac{1}{t_k}
\leq \frac{t_k}{t_{k+1}} - \frac{1}{t_{k+1}} = \frac{t_k - 1}{t_{k+1}}, \quad
k=1,2,\dotsc.
\]
Since $t_{k+1} \geq t_k$ (which follows immediately from
\eqref{eq:nesseq}), it is sufficient to prove that
\[
\frac{t_{k-1}}{t_k} \leq \frac{t_k}{t_{k+1}}.
\]
By manipulating this expression and using (\ref{eq:tsquared}), we
obtain the equivalent expression
\begin{equation} \label{eq:gs1}
t_{k-1} \leq \frac{t_k^2}{t_{k+1}} = \frac{t_{k+1}^2 - t_{k+1}}{t_{k+1}} = t_{k+1} - 1.
\end{equation}
By definition of $t_{k+1}$, we have
\[
t_{k+1} = \frac{\sqrt{4t_k^2 + 1} + 1}{2} \geq t_k + \frac{1}{2}
= \frac{\sqrt{4t_{k-1}^2 + 1} + 1}{2} + \frac{1}{2} \geq t_{k-1} + 1.
\]
Thus \eqref{eq:gs1} holds, so the claim \eqref{eq:gt2} is proved. The
sequence $\{ (t_{k-1}-1)/t_k \}$ is nonnegative, since $(t_0-1)/t_1
= 0$.

Now we prove \eqref{eq:gt3}. We can lower-bound $(t_{k-1}-1)/t_k$ as follows:
\begin{align} 
\nonumber
& \frac{t_{k-1}-1}{t_k} = \frac{2(t_{k-1} -1)}{\sqrt{4 t_{k-1}^2 +1} + 1} 
\label{eq:uy1}
\geq \frac{2(t_{k-1} -1)}{\sqrt{4 t_{k-1}^2} + 2} \\
& \quad\quad\quad\quad = \frac{2(t_{k-1} -1)}{2 (t_{k-1} + 1)}
= 1 - \frac{2}{t_{k-1}+1}.
\end{align}
For an upper bound, we have from $t_k \geq t_{k-1}$ that
\begin{equation} \label{eq:uy2}
\frac{t_{k-1} -1}{t_k}  \le \frac{t_{k-1}}{t_k}  \le 1.
\end{equation}
Since $t_{k-1} \to \infty$ (because of \eqref{eq:tlb}), it follows
from \eqref{eq:uy1} and \eqref{eq:uy2} that \eqref{eq:gt3} holds.

\end{document}